\documentclass{amsart}
\usepackage{amsmath,amssymb,verbatim}
\usepackage[utf8]{inputenc}
\pdfoutput=1 
\usepackage{enumerate}
\usepackage{pdflscape}
\usepackage{amsthm}
\usepackage{mathrsfs}
\usepackage{color}
\usepackage{multicol}
\usepackage[normalem]{ulem}
\usepackage{cancel}
\usepackage{tikz}
\usetikzlibrary{matrix}
\usepackage[all]{xy}
\usepackage{mathtools}
\usepackage{arydshln}
\usepackage[shortlabels]{enumitem}
  \usepackage[english]{babel}
\usepackage[capitalize, noabbrev]{cleveref}
 
\makeatletter
\@namedef{subjclassname@1991}{\textup{2020} Mathematics Subject Classification}
\makeatother

\topskip=\baselineskip 
\newtheorem{theorem}{Theorem}[section]
\newtheorem{lemma}[theorem]{Lemma}

\newtheorem{proposition}[theorem]{Proposition}

\newtheorem{corollary}[theorem]{Corollary}
\newtheorem{remark}[theorem]{Remark}

\newtheorem{question}{Question}

\newtheorem*{theorem*}{Theorem}

\makeatletter
\@addtoreset{equation}{section}
\makeatother

\newcommand{\Q}{{\mathbb Q}}

\newcommand{\F}{{\mathbb F}}
\newcommand{\HQ}{{\mathbb H}}

\newcommand{\GEN}[1]{\left\langle #1 \right\rangle}

\newcommand{\ZZ}{\mathrm{Z}}
\newcommand{\U}{\mathcal{U}}

\newcommand{\qand}{\quad \text{and} \quad}

\title[]{On commutative tensor factors of group algebras}

\author{Diego Garc\'{\i}a-Lucas }
 \address{Departamento de Matem\'aticas, Universidad de Murcia, Spain}
 \email{diego.garcial@um.es}

\author{Ángel del Río}
 \address{Departamento de Matem\'aticas, Universidad de Murcia, Spain}
 \email{adelrio@um.es}

 \author{Taro Sakurai}
 \address{Department of Mathematics and Informatics, Graduate School of Science, Chiba University, 1-33, Yayoi-cho, Inage-ku, Chiba-shi, Chiba, 263-8522, Japan}
 \email{tsakurai@math.s.chiba-u.ac.jp}
 
\thanks{The first two authors have been partially supported by Grant PID2020-113206GB-I00 funded by MCIN/AEI/10.13039/501100011033  and by Grant Fundación Séneca 22004/PI/22. }

\keywords{group algebras, tensor factors.}

\subjclass{16S34, 20C05, 16L30}

\date{\today}

\begin{document}

\begin{abstract}  
 We prove that any tensor product factorization with a commutative factor of a modular group algebra over a prime field comes from a direct product decomposition of the group basis. This extends previous work by Carlson and Kovács for the commutative case and answers a question of them in some cases. 
\end{abstract}

 \maketitle
  
 The aim of this paper is to investigate how a group algebra $kG$ of a group $G$ over a field $k$ may decompose as tensor product of subalgebras. The most obvious factorizations come from direct product decompositions of $G$ but, in general, these are not the only ones. 
For example, if $k$ has an element of multiplicative order $4$ then $kC_4\cong k(C_2\times C_2)\cong kC_2\otimes_k kC_2$, even though $C_4$ is indecomposable as a direct product.	
More elaborated examples in the semisimple case are provided at the end of this paper. 
However, no example of tensor product factorization not coming from a direct product decomposition of the group basis is known in the modular case, i.e. when $G$ is a finite $p$-group with $p$ the characteristic of $k$. 
If it were true that every tensor product factorization of a modular group algebra comes from a direct product factorization of the group basis, then there would be a Krull-Schmidt-like theorem for tensor factorizations of modular group algebras. 
Observe that modular group algebras are finite dimensional, local and augmented; and for artinian local augmented algebras of characteristic zero such Krull-Schmidt property holds \cite{Horst1986}.  

 A first attempt to study tensor product factorizations of modular group algebras was carried out by J. F. Carlson and L. G. Kovács who proved that for commutative modular group algebras any tensor product decomposition comes from a direct product decomposition of the group basis. More precisely the prove the following theorem. 

 \begin{theorem}\cite[Theorem 2.2]{CarlsonKovacs} \label{CarlsonKovacs}
  Let $p$ be a prime, $k$ be a field of characteristic $p$, and $G$ a finite abelian $p$-group.  To each tensor factorization $kG \cong \bigotimes_{i=1}^n A_i$ of $kG$, there is a direct decomposition $G=\prod_{i=1}^n G_i$ such that for each $i$:
  \begin{enumerate}
   \item $A_i\cong kG_i$, and
   \item $kG= A_i \otimes \left( \bigotimes_{j\neq i} kG_j\right)$.
  \end{enumerate}
 \end{theorem}

 We continue this study by proving the following result, which extends the first item in \Cref{CarlsonKovacs} for group algebras over the prime field. 
 
 \begin{theorem}\label{theorem}
 	Let $G$ be a finite $p$-group, let $\F_p$ be the field of $p$ elements, and suppose that $\F_p G=B\otimes_{\F_p} C$  for some subalgebras $B$ and $C$ of $\F_p G$ such that $B$ is commutative. Then $G=\mathcal B\times  \mathcal C$ for some subgroups $\mathcal B$ and $\mathcal C$ of $G$ such that $\F_p \mathcal B\cong B$ and $\F_p \mathcal C\cong C$.
\end{theorem}

 As the ``simplest'' question that  arise from \Cref{CarlsonKovacs}, Carlson and Kov\'acs ask the following: 

\begin{question}\cite[Section 5]{CarlsonKovacs}\label{CKQuestion}
Does the group algebra of a direct indecomposable $p$-group over a field of characteristic $p$ ever admit a nontrivial tensor factorization (as algebras)?
\end{question}

In \cite[Section 5]{CarlsonKovacs} they, in an initial attempt to tackle this question in the noncommutative case, answer their question   for the nonabelian groups of order $8$, and, to the best of our knowledge, no further progress on this topic has been made since. Their argument relies on the product formula of centers $Z(B \otimes C) = Z(B) \otimes Z(C)$. Another product formula 
%\ChAngel{$(B \otimes C)/[C, C](B\otimes C) \cong B \otimes C/[C, C]C$ 
%\sout{$(B \otimes C)/[B \otimes C, B \otimes C] \cong B/[B, B] \otimes C/[C, C]$}} 
%\ChAngel{[Observe that the second formula is a consequence of the first and in fact the first one is the one we are using]} 
 $[B\otimes C, B\otimes C] = B\otimes [C,C] + [B,B]\otimes C$ (as submodules of $A\otimes B$, see \cite{LinckelmannBook1})  may be used instead here. As an application of \Cref{theorem}, we give answer to   Carlson--Kov\'acs  question for    groups generated by at most $3$ elements and for groups with cyclic derived subgroup.

\begin{corollary}\label{corollary}
Let $G$ be a finite direct indecomposable $p$-group and let $\mathbb F_p$ be the field of $p$ elements. Then $\mathbb{F}_p G$ is indecomposable as a tensor product of proper subalgebras provided at least one of the following holds:
\begin{enumerate}
 \item $G$ can be generated by  $3$ elements,
 \item  $G'$ is cyclic.
\end{enumerate}
\end{corollary}
%  \ChTaro{
%    I would formulate these similar to above theorems: ``if $\F_p G = B \otimes C$ then $B \cong \F_p H$, $C \cong \F_p K$, and $G = H \times K$ for some $H, K \le G$ \dots'' for three reasons.
%    1. Proof goes direct way. Needless proof by contradiction can be omitted.
%    2. We don't need to exclude the trivial tensor factor $\F_p$ separately.
%    3. We don't need to exclude the trivial group algebra $\F_p$, not indecomposable, separately.
%}
%\ChAngel{[We prefer to leave it as it is for the following reasons: It answers directly \Cref{CKQuestion}, the two versions are equivalent and, the alternative version does not look sufficiently different to \Cref{theorem}.]}
Observe that our result includes many finite $p$-groups, such as those that are metacyclic  , extra-special, of maximal class, or of order at most $p^5$.

\bigskip
 We start setting basic notation. Throughout 
  $p$ is a prime number, $k$ is a field of characteristic $p$, $G$ is a finite $p$-group and $C_n$ denotes a cyclic group of order $n$. 
From \Cref{lemma4.11} on, $k$ will be always the field $\F_p$ with $p$ elements. 
By convention, $\otimes$ means $\otimes_k$.

 Let $A$ be a $k$-algebra. The group of units of $A$ is denoted $\U(A)$. 
If $X$ and $Y$ are subsets of a $k$-algebra $A$ then  $XA$ denotes the right ideal of $A$ generated by $X$ and  $[X,Y]$ denotes the $k$-span of the set of Lie commutators $[x,y]=xy-yx$ with $x\in X$ and $y\in Y$.

Recall that an augmented $k$-algebra is a pair formed by a unital $k$-algebra and an ideal $I(A)$ of $A$ of codimension $1$, called augmentation ideal of $A$.
 We write $V(A)=\U(A)\cap (1+I(A))$, a normal subgroup of $\U(A)$.
%\ChTaro{Either $\mathcal{V}(A)$ and $\mathcal{U}(A)$, or $V(A)$ and $U(A)$.}
%\ChAngel{[The two notations are well established in many references in the field, so we prefer not to change it, but we understand your point, so if you insist we will accept your proposal]}
If $I(A)$ is nilpotent, then $V(A)$ is a $p$-group of exponent $p^s$, where $s$ is the smallest non-negative integer with $a^{p^s}=0$ for every $a\in I(A)$.
Moreover, every subalgebra $B$ of $A$ is augmented with augmentation ideal $I(B)=B\cap I(A)$, and if $A=B\otimes C$ with $B$ and $C$ subalgebras of $A$ then 
  \begin{equation}\label{decomposition}
A=k\oplus I(B)\oplus I(C)\oplus I(B) I(C)
\qand I(A)=I(B)\oplus  I(C)\oplus I(B)I(C).
\end{equation} 

For example, the group algebra $kG$ is an augmented algebra and its  augmentation ideal is its Jacobson radical, which is nilpotent and generated by the elements of the form $g-1$ with $g\in G$.
We abbreviate $I(kG)$ by $I(G)$. Observe that if $N$ is a normal subgroup of $G$ then $I(N)kG$ is the kernel of the natural homomorphism $kG \to k(G/N)$.

  We use the notation  $\ZZ(G)$ for the center of $G$, and for a positive integer $i$, 
  $$\Omega_i(G)=\GEN{g\in G : g^{p^i}=1}, \quad \mho_i(G)=\GEN{g^{p^i}: g\in G} $$
and
  $$R_i(G)=\frac{\Omega_i(\ZZ(G)) \mho_i(G) G'}{\mho_i(G) G'}.$$   
We also use analogous versions of the first   three notations in the context of algebras, namely, for a non-necessarily unital  $k$-algebra $A$,  $\Omega_i(A)$ denotes the subalgebra of $A$ generated by the elements $a\in A$ such that $a^{p^i}=0$, and $\mho_i(A)$ denotes the subalgebra of $A$ generated by the elements of the form $a^{p^i}$ with $a\in A$.

Let $H_i(G)$ denote a homocyclic component of $G$ of exponent $p^i$, in the sense of \cite{GL2022}, i.e. $H_i(G)$ is maximal among the direct factors of $G$ that are isomorphic to a direct product of cyclic groups of order $p^i$.
This is unique up to isomorphism and we abuse notation by referring to it as ``the'' homocyclic component of $G$.

  \begin{lemma} \label{lemma:Ideals} For every pair of positive integers $i$  and $j$  we have
  	\begin{enumerate}
  		\item\label{lemma:Ideals1}
  		$I(\mho_i(G)\;G')kG = \mho_i(I(G))kG+I(G')kG$.
%
%   		$I(\mho_i(G)G')kG$ is the ideal of $kG$ generated by
%   		$$ \{ z^{p^i} :z\in I( G)\} \cup I(G')\ChAngel{\cancel{kG}}.$$

\item \label{IdealsUnoymedio}
$I(\Omega_i(\ZZ(G))G')kG = 	\Omega_i(\ZZ(I(G)))kG + I(G')kG $.

  		\item\label{lemma:Ideals2}
  		$I(\Omega_i(\ZZ(G)) \; \mho_j(G) \;G')kG =
  		\Omega_i(\ZZ(I(G)))kG + \mho_j(I(G))kG + I(G')kG$.

%   		$
%   		I(\Omega_i(\ZZ(G)) \mho_j(G) G')kG$ is the ideal generated by
%   		$$ \{ z^{p^i} :z\in G\} \cup G'\ChAngel{)\cancel{kG}} \cup \{  z \in \ChAngel{\cancel{\ZZ(kG)\cap G)} \ZZ(I(G))} : z^{p^i}=0\}.$$

	\item \label{lemma:Ideals3}
  	$G$ has a cyclic direct factor of order $p^i$ if and only if
  	$$ \exp\left( 1+  I\left( R_i(G)  \right) k \left(  \frac{G}{\mho_i(G)G'}  \right) \right)\geq p^i.$$
  	\end{enumerate}
  \end{lemma}  

\begin{proof}
%The proof of the first two is inspired by that of \cite[Lemma 3.6]{GL2022}.

% \eqref{lemma:Ideals1} Since $I(G')kG$ is contained in both ideals, we may quotient it out and assume that $G$ is abelian. For an abelian group it is clear that $I(\mho_i(G)) =\mho_i(I(G))$.

\eqref{lemma:Ideals1} is \cite[Lemma~2.6]{MSS23} and \eqref{IdealsUnoymedio} is a particular case of \cite[Lemma~2.3]{MSS23}.
  	
\eqref{lemma:Ideals2} 
 Let $\pi:kG\to k(G)/(I(\mho_j(G))kG+I(G')kG)$ be the natural homomorphism and $\lambda:k Z(G) \to k Z(G)$ be the map given by $\lambda(x)=x^{p^i}$. By \cite[Lemma 6.10]{Sandling85},
$$\ZZ(I(G)) = I(\ZZ(G))\oplus [kG,kG] \cap \ZZ(kG)$$
and, by \cite[Proposition III.6.1]{Seh78},
	$$\ker(\lambda)= I(\Omega_i(\ZZ(G)))k\ZZ(G).$$ 
Thus 
\begin{eqnarray*}
\pi(\Omega_i(Z(I(G)))) &=& \pi(\Omega_i(I(\ZZ(G))\oplus [kG,kG] \cap \ZZ(kG)))
= \pi(\Omega_i(I(\ZZ(G)))\oplus \Omega_i([kG,kG] \cap \ZZ(k(G)))) \\
&=& \pi(\Omega_i(I(\ZZ(G)))) = \pi(\ker \lambda) = \pi(I(\Omega_i(\ZZ(G)))k\ZZ(G)),	
\end{eqnarray*}
and the result follows.  

\eqref{lemma:Ideals3}
By \cite[Lemma 4.3]{GL2022}, $H_i(G)\cong H_i ( R_i(G))$.
  Observe that the group $R_i(G)$ is abelian of exponent at most $p^i$. Hence $G$ has an abelian direct factor  of exponent $p^i$ if and only if $\exp(R_i(G))\geq p^i$. Now, since $G/\mho_i(G)G'$ is abelian, we have that
  $$\exp\left( 1+  I\left( R_i(G)  \right) k \left(  \frac{G}{\mho_i(G)G'}  \right)  \right)= \exp(R_i(G)).$$
  \end{proof}

\begin{proposition} \label{prop:Babelian}
	  Suppose that $kG=B\otimes C$, where $B$ and $C$ are proper subalgebras of $kG$ with $B$ commutative. Then
\begin{enumerate}[(a)]
\item \label{eq:ExpansionCommutator}
   $[kG,kG] =[I(C),I(C)] + [I(C),I(C)]I(B) \subseteq I(C)\oplus I(B)I(C)$.
\item \label{prop:Babelian1} $B$,   $C/[C,C]C$  and $kG/[C,C]kG$  are group algebras. 
\item \label{Bembeds} If $p^s$ is the exponent of $V(B)$, then $I(\mho_s(G)G')kG \subseteq I(C)\oplus I(B)I(C)$. 
\item \label{prop:Babelian2}
	 $G$ has a cyclic direct factor of order $\exp(V(B))$.
\end{enumerate}
\end{proposition}

\begin{proof}
\ref{eq:ExpansionCommutator} Since $B$ is central in $kG$, by equation \eqref{decomposition},
  	\begin{align*}
  		[kG,kG]% &=[k\oplus I(B)\oplus I(C)\oplus I(B)I(C),  k\oplus I(B)\oplus I(C)\oplus I(B)I(C)]  \\
  		&
  		= [I(C), I(C)]+ [I(C), I(B)I(C)]+ [I(B)I(C), I(B)I(C)]  \\
  		&=[I(C),I(C)] + [I(C),I(C)]I(B)\subseteq I(C)\oplus I(B)I(C).
  	\end{align*}
%
% $$[kG,kG] =   [I(C), I(C)] + [I(C), I(B)I(C)]+ [I(B) I(C), I(B)I(C)]
%   \subseteq I(C)\oplus I(B)I(C).$$

\ref{prop:Babelian1}  
It is well-known that $k(G/G')\cong kG/I(G')kG$ and $[kG,kG]kG= I(G')kG$. 
On the other hand, there is a natural epimorphism
  	$B\otimes C \to B \otimes (C/[C,C]C)$ given by $b\otimes c\mapsto b\otimes(c+[C,C]C)$. Since its kernel is   $[C,C]C$ and the target algebra is commutative, $[kG,kG]kG=[C,C]C=[C,C]kG$. Thus $k(G/G')\cong kG/[C,C]kG \cong B\otimes C/[C,C]C$.  Then the result follows from \Cref{CarlsonKovacs}.

\ref{Bembeds}
Let $p^s$ be the exponent of $V(B)$. 
Since $B$ is central in $kG$ and $I(C)\oplus I(B)I(C)$ is an ideal of $A$, $z^{p^s}\in I(C)\oplus I(B)I(C)$ for every $z\in I(  G )$. 
This, \ref{eq:ExpansionCommutator} and \Cref{lemma:Ideals}\eqref{lemma:Ideals1} show  that $I(\mho_s(G)G')kG \subseteq I(C)\oplus I(B)I(C)$. 

\ref{prop:Babelian2}
Let $x\in I(B)$ such that $1+x$ reaches the exponent of $V(B)$, i.e., $|1+x|=p^s$.
By  \eqref{decomposition} and \eqref{Bembeds}, $B\cap I(\mho_s(G)G')kG=0$, so
$$|1+x+I(\mho_s(G)G')kG|=p^s.  $$
On the other hand,
	$$x\in \{  z\in \ZZ(kG)\cap I(G) : z^{p^s}=0\} \subseteq I(\Omega_s(\ZZ(G))\mho_s(G)G')kG,$$
by \Cref{lemma:Ideals}\eqref{lemma:Ideals2},
so
$$1+x+I(\mho_s(G)G')kG \in 1+\frac{I(\Omega_s(\ZZ(G)) \mho_s(G) G')kG}{I(\mho_s(G) G')kG}\cong 1+  I\left( R_s(G) \right) k \left(  \frac{G}{\mho_s(G)G'}  \right). $$
Thus   the latter  group has an element of order $p^s$, hence  \Cref{lemma:Ideals}\eqref{lemma:Ideals3} yields the result.
\end{proof}

%
% A much stronger consequence of a positive answer to \Cref{Question:ysinoBabelian1} is the following.
%
% \begin{remark}
% A positive answer to \Cref{Question:ysinoBabelian1} implies a Krull-Schmidt type property for tensor factorizations of group algebra, namely: if $kG=A_1\otimes \dots \otimes A_n=B_1\otimes \dots \otimes B_m$, with each $A_i, B_j$ tensor indecomposable, then $n=m$ and $A_i\cong B_i$ for $i=1,\dots, n$, after possibly a reordering of the indices of the $B_i$'s. Indeed, if $kG=A_1\otimes A_2=B_1\otimes B_2 $, then
% \end{remark}

In the remainder of the paper $k=\F_p$ and consequently $I(H)$ denotes the augmentation ideal of $\F_p H$, for each finite $p$-group $H$. 

 \begin{lemma}\cite[Lemma 4.11]{GL2022}\label{lemma4.11} Let $V$ be a subspace of $I(G)$ containing $I(G)^2$, and consider the map
\begin{align*}
	\Lambda^{s-1}_G: \frac{I(\Omega_s(\ZZ(G))G')\mathbb F_pG+I(G)^2}{I(G)^2} &\to \frac{I(G)^{p^{s-1}}+ I(\mho_s(G)G')\mathbb F_p G}{I(G)^{p^{s-1}+1}+ I(\mho_s(G)G')\mathbb F_p G} \\
	 z+ I(G)^2&\mapsto z^{p^{s-1}}+I(G)^{p^{s-1}+1}+ I(\mho_s(G)G')\mathbb F_p G.
\end{align*}
Then the following conditions are equivalent:
\begin{enumerate}
 \item There is a direct sum decomposition
 $$ \frac{V}{I(G)^2} \oplus \ker \Lambda_G^{s-1}= \frac{I(\Omega_s(\ZZ(G))G')\mathbb F_pG+I(G)^2}{I(G)^2}.$$
 \item There exists a decomposition $G=H\times K$ with $H$ homocyclic of exponent $p^s$ and $K$ not admitting any cyclic direct factor of order $p^s$ such that
 $$V=I(H)\mathbb F_pG+I(G)^2.  $$
\end{enumerate}

 \end{lemma}

 We are finally ready to prove the main theorem.

\begin{proof}[Proof of \Cref{theorem}]
By \Cref{prop:Babelian}\ref{prop:Babelian1}, there is a group basis
$\mathcal B$ of $B$.
Let $p^s$ be the exponent of $\mathcal B$.
Since $B$ is commutative, $V(B)$ has exponent $p^s$ too.
As $B$ is central in $\F_p G$, \Cref{lemma:Ideals}\eqref{IdealsUnoymedio} yields  $I(B)\subseteq I(\Omega_s(\ZZ(G))G')\mathbb F_pG$.

Let $b\in \mathcal B$ be an element of order $p^s$ in $\mathcal B$.
So $\mathcal B=\GEN{b}\times \mathcal B_0$ for some subgroup $\mathcal B_0$ of $\mathcal B$. 
Then $b-1\in I(B)\setminus I(B)^2$ and  $(b-1)^{p^{s-1}}\in I(B)^{p^{s-1}}\setminus  I(B)^{p^{s-1}+1}$, by Jennings' Theorem (see   \cite{Lazard1953} or \cite[Theorem~11.1.20]{Pas77}).

We claim that $(b-1)^{p^{s-1}}\not\in I(G)^{p^{s-1}+1}+I(\mho_s(G)G')\F_p G$. 
Indeed, the decomposition $I(G)=I(B)\oplus I(C)\oplus I(B)I(C)$ yields a surjective algebra homomorphism $\pi:I(G)\to I(B)$ with kernel $I(C)\oplus I(B)I(C)$ and $\pi(I(G)^m)=I(B)^m$ for any positive integer $m$. If $(b-1)^{p^{s-1}}\in I(G)^{p^{s-1}+1}+I(\mho_s(G)G')\F_p G$ then 
there is $b_1\in I(B)^{p^{s-1}+1}$ such that $(b-1)^{p^{s-1}}-b_1\in I(C)\oplus I(B)I(C) + I(\mho_s(G)G')\F_p G\subseteq I(C)\oplus I(B)I(C)$, by \Cref{prop:Babelian}\ref{Bembeds}. 
Then $(b-1)^{p^{s-1}}=b_1\in I(B)^{p^{s-1}+1}$, in contradiction with the previous paragraph. 

Then $b-1+I(G)^2\not\in \ker \Lambda_G^{s-1}$.  
Therefore there is a subspace $V$ of $I(G)$ containing $b-1$ and $I(G)^2$ such that 	
	$$\frac{I(\Omega_s(\ZZ(G))G')\mathbb F_pG+I(G)^2}{I(G)^2}=\frac{V}{I(G)^2}\oplus \ker\Lambda_G^{s-1}.$$  
Then  \Cref{lemma4.11} yields that $G= H\times K$, with $H$ homocyclic of exponent $p^s$ and $I(H)\F_pG+I(G)^2=V$. 
 Let $\Phi(G)$ denote the Frattini subgroup of $G$. 
As $g\Phi(G)\mapsto g-1+I(G)^2$ defines an isomorphism from $G/\Phi(G)$ to the additive group of $I(G)/I(G)^2$ \cite[Proposition~III.1.15]{Seh78}, there is $h\in H$ such that $h-1+I(G)^2=b-1+I(G)^2$.  
Then $h-1+I(G)^2\not\in \ker \Lambda_G^{s-1}$ and therefore $h$ has order $p^s$. 
Thus $G=\GEN{h}\times G_0$ for some subgroup $G_0$ of $G$.  

Let $J$ be the ideal of $\F_pG$ generated by $b-1$.
Thus $J+I(G)^2=I(\GEN{h})\F_p G+I(G)^2$.
Write $C_0=\F_p\mathcal B_0 \otimes C $, so
	$$\F_pG=\F_p\GEN{b}\otimes  C_0=\F_p\oplus I(\GEN{b})\oplus I(C_0)\oplus I(\GEN{b})I(C_0)= C_0 \oplus J.$$
Thus the codimension of $J$ in $\F_pG$ is $\dim(C_0)=|G|/|b|=|G|/|h|=|G_0|$.
Now it follows from \cite[Lemma 4.9]{GL2022} that
$$\F_pG=J\oplus \F_pG_0. $$
Thus $C_0\cong \F_pG_0$.
Proceeding by induction on the size of $\mathcal B$, we derive that $C\cong \F_p \mathcal C$ for some subgroup $\mathcal C$ of $G$ such that $G\cong \mathcal B\times \mathcal C$. 
This finishes the proof of \Cref{theorem}.
 \end{proof}

\begin{remark}{\rm 
Observe that the assumption $k=\F_p$ has been used to apply \Cref{lemma4.11} and  to use the isomorphism $G/\Phi(G)\cong I(G)/I(G)^2$. We wonder whether \Cref{theorem} holds for arbitrary coefficient fields of characteristic $p$.}
\end{remark}

%In the form of a theorem: 
%\begin{theorem}
%	Let $G$ be a finite $p$-group and let $k$ the field of $p$ elements. If $kG=B\otimes_k C$ for some subalgebras $B$ and $C$ of $kG$ with $B$ commutative, then $G=\mathcal B\times \mathcal G$ for some subgroups subgroups $\mathcal B$ and $\mathcal G$ of $G$, with $\mathcal B$ abelian and such that   $B\cong k\mathcal B$ and $C\cong k\mathcal G$.
%\end{theorem}

%\section{Consequences: Carslon and Kovacs question for some classes of $p$-groups}
% Now we apply the previous theorem to study the following question
%
% \begin{question}\cite[Section 5]{CarlsonKovacs}
% Does the group algebra of a direct indecomposable $p$-group over a field of characteristic $p$ ever admit a nontrivial tensor factorization (as algebras)?
% \end{question}
%
% which we can prove over the prome field for groups that are generated by at most $3$ elements.

We split the proof of the two statements of \Cref{corollary}.

\begin{corollary}
Let $G$ be a finite direct indecomposable $p$-group. If $G$ is at most $3$-generated, then $\F_p G$ is indecomposable as tensor product of proper subalgebras.
\end{corollary}
\begin{proof}
Suppose that a nontrivial tensor factorization $\F_p G = B \otimes C$ exists.
As $I(B)\cap I(G)^2=I(B)^2$, and similarly with $C$, 
$$I(G)/I(G)^2= (I(B)+I(G)^2)/I(G)^2\oplus (I(C)+I(G)^2)/I(G)^2\cong 
I(B)/I(B)^2 \oplus I(C)/I(C)^2.$$
Moreover, $I(B)^2\ne I(B)$, and similarly with $C$, so the two summands are non-trivial. 
As $G$ is generated by (at most) $3$ elements, $|G/\Phi(G)|=|I(G)/I(G)^2|\le p^3$, and hence one of the summands in the above decomposition has dimension $1$, say the first one. 
Then $B=B_0+I(B)^2$ for any unital algebra $B_0$ generated by one element in $I(B)\setminus I(B)^2$. 
Then $B=B_0$, by \cite[Proposition~5.2]{Kulshammer1991}. 
So applying \Cref{theorem} we derive that $B$ is isomorphic to some group algebra $\F_p H$, where $H$ is a direct factor of $G$, a contradiction. 

\end{proof}

%The following elementary fact will be useful.
%
%\begin{lemma}
%Let $J_1\subseteq J_2$ be ideals of $kG$. If $J_1+J_2I(G)=J_2$ then $J_1=J_2$.
%\end{lemma} 
%
%\begin{proof}
%As $I(G)$ is the Jacobson radical of $kG$, $J_2I(G)$ is contained in the radical of $J_2$ and hence it is superfluous in $J_2$.
%\end{proof}

\begin{corollary}
	Let $G$ be a finite direct indecomposable $p$-group. If $G'$ is cyclic, then $\F_p G$ is indecomposable as tensor product of proper subalgebras.
\end{corollary}

\begin{proof}
Suppose by contradiction that $\F_p G=B\otimes C$, with both $B$ and $C$ nontrivial subalgebras of $\F_p G$. 
%\sout{
%As $G'$ is cyclic, $I(G')\F_p G$ is generated as right ideal of $\F_p G$ by an element of the form $\alpha=g^{-1}h^{-1}gh-1$ with $g,h\in G$, and this implies that $I(G')\F_p G/I(G')I(G)$ is one dimensional. 
%}
As $G'$ is cyclic, it follows from \cite[Proposition III.1.15(ii)]{Seh78} that $I(G')\F_p G/I(G')I(G)$ is one-dimensional.
Moreover, $I(G')\F_p G=[I(G),I(G)]\F_p G$ and hence $I(G')\F_p G/I(G')I(G)$ is span by any non-zero element of the form $[a_1,a_2]+I(G')I(G)$ with $a_1,a_2\in I(G)$.   Since $I(G)=I(B)\oplus I(C)\oplus I(B)I(C)$, $[I(G),I(G)^2]\subseteq I(G')I(G)$ and $[I(B),I(C)]=0$, we can take $a_1$ and $a_2$ as above, with both in $I(B)$ or both in $I(C)$. Without loss of generality assume that they belong to $I(C)$.
But $I(C)\oplus I(B)I(C)=I(C)\F_p G$, so $I(G')\F_p G=[a_1,a_2]\F_p G  \subseteq I(C)\oplus I(B)I(C)$. It follows that $[I(B),I(B)]\subseteq I(B)\cap I(G')\F_p G\subseteq I(B)\cap (I(C)\oplus I(B)I(C))=0$, so $B$ is commutative. Hence $G$ is decomposable by \Cref{theorem}, a contradiction.
\end{proof}

%\begin{proof}
%		In this proof we consider  the Lie commutator map 
%	\begin{align*}
%		[,]: \frac{I(G)}{I(G)^2}\times \frac{I(G)}{I(G)^2}&\to \frac{I(G')\F_p G}{I(G')I(G)} \\
%		(x+I(G)^2,y+I(G)^2)&\mapsto xy-yx+I(G')I(G),
%	\end{align*}  
%	which is clearly well-defined. 
%	
%	
%	Since $G$ is cyclic-by-abelian, $G'$ is cyclic, and hence $G'/\Phi(G')$ has order $p$. Since $I(G')\F_p G/I(G')I(G)\cong G'/\Phi(G')$, the first quotient has dimension $1$ over $\F_p $.   Thus there are elements $a,b\in I(G)$ such that $[a+I(G)^2, b+I(G)^2]$ generates $I(G')\F_p G/I(G')I(G)$. In particular $ab-ba\neq 0$, and, by the previous lemma, $ab-ba$ generates $I(G')\F_p G$ as an ideal of $\F_p G$.  
%	
%	
%Suppose by contradiction that $\F_p G=B\otimes C$, with both $B$ and $C$ nontrivial subalgebras of $\F_p G$. Then $$\frac{I(G)}{I(G)^2} = \frac{I(B)+I(G)^2}{I(G)^2}\oplus \frac{I(C)+I(G)^2}{I(G)^2}. $$ Since the elements of $B$ commute with the elements of $C$, we can assume that both $a$ and $b$ belong to $I(B)$.  Therefore $[a,b]$ belongs to $I(B)$. Cut $I(B)\oplus I(B)I(C)$  is an ideal of $\F_p G$, so $I(G')\F_p G$, the ideal generated by $[a,b]$, is contained in $I(B)I(C)$. It follows that $[I(C),I(C)]\subseteq I(C)\cap I(G')\F_p G\subseteq I(C)\cap (I(B)\oplus I(B)I(C))=0$, so $C$ is commutative. Hence $G$ is decomposable by \Cref{theorem}.
%\end{proof}

We wonder whether the commutativity assumptions in both \Cref{theorem}
\Cref{prop:Babelian}\ref{prop:Babelian1} is really needed.

\begin{question}\label{Question:ysinoBabelian1}
Suppose that $\F_pG=B\otimes C$.
\begin{enumerate}
\item\label{Question1} Is $B$ a group algebra?
\item\label{Question2} Is $\F_pG/[C,C]\F_pG$ a group algebra?
\end{enumerate}
\end{question}

Actually, it is easy to see that the two questions are equivalent.
Indeed, if the answer to the first question is positive, then $B$ and $C$ are group algebras and hence so are $C/[C,C]C$ and $B\otimes C/[C,C]\cong \F_G/[C,C]\F_p G$. 
Conversely, if the second question has positive answer, then $\F_pG/[C,C]\F_pG$  is a group algebra, which is isomorphic to $B\otimes C/[C,C]C$. Hence  $B$ is also a group algebra, by \Cref{theorem}.

Observe that in the non-modular case over fields it is easy to find non-trivial factorizations of group algebras of indecomposable groups. For example, if $q$ is a divisor of $p-1$ then all the $\F_p$-group algebras of abelian groups of order $q$ are isomorphic. In particular, if $q=r^n$ with $r$ prime and $n\ge 2$, then $C_q$ is indecomposable while $\F_p C_q$ have non-trivial tensor factorizations. This example also shows that \Cref{theorem} fails in the non-modular case.
Observe that the tensor factors are group algebras since the example is built upon the failing of the isomorphism problem in this setting, even existing an indecomposable group and a decomposable one with isomorphic group algebras.
This may suggest that the tensor factors of group algebras should be group algebras again, and in that case group algebras of indecomposable groups with positive answer for the isomorphism problem should be tensor indecomposable. Unfortunately this is not the case. For example, the isomorphism problem has a positive solution of rational group algebras of groups of order 16. This can be seen by simply examining the Wedderburn decomposition of the rational group algebras of the non-abelian groups of order $16$:
% \begin{center}
% \begin{tabular}{|c|cccccccc|}
% \hline
% id & $\Q$ & $\Q(i)$ & $M_2(\Q)$ & $\HQ(\Q)$ & $M_2(\Q(i))$ & $M_2(\Q(\sqrt{2}))$ & $M_2(\Q(\sqrt{-2}))$ & $M_2(\HQ(\Q(\sqrt{2}))$ \\\hline
% 3 & 4 & 2 & 2 & 0 & 0& 0& 0 & 0\\
% 4 & 4 & 2 & 1 & 1 & 0& 0& 0 & 0\\
% 6 & 4 & 2 & 0 & 0& 1& 0 & 0 & 0\\
% 7 & 4 & 0 & 1 & 0& 0& 1 & 0 & 0\\
% 8 & 4 & 0 & 1 & 0& 0& 0 & 1 & 0\\
% 9 & 4 & 0 & 1 & 0& 0& 0 & 0 & 1\\
% 11 & 8 & 0 & 2 & 0& 0& 0 & 0 & 0\\
% 12 & 8 & 0 & 0 & 2& 0& 0 & 0 & 0\\
% 13 & 8 & 0 & 0 & 0& 1& 0 & 0 & 0\\\hline
% \end{tabular}
% \end{center}
\begin{eqnarray*}
\Q G_3 &=& 4\Q \oplus 2\Q(\sqrt{-1}) \oplus 2M_2(\Q) \\
\Q G_4 &=& 4\Q \oplus 2\Q(\sqrt{-1}) \oplus 2\HQ(\Q) \\
\Q G_6 &=& 4\Q \oplus 2\Q(\sqrt{-1}) \oplus M_2(\Q(\sqrt{-1})) \\
\Q G_7 &=& 4\Q \oplus M_2(\Q) \oplus M_2\Q(\sqrt{2})  \\
\Q G_8 &=& 4\Q \oplus M_2(\Q) \oplus M_2\Q(\sqrt{-2})  \\
\Q G_9 &=& 4\Q \oplus M_2(\Q) \oplus  \HQ(\Q(\sqrt{2})) \\
\Q G_{11} &=& 8\Q \oplus 2M_2(\Q) \\
\Q G_{12} &=& 8\Q \oplus 2\HQ(\Q) \\
\Q G_{13} &=& 8\Q \oplus M_2(\Q(\sqrt{-1}))
\end{eqnarray*}
Here $G_i$ represents the $i$-th group in the {\sf GAP} Small Groups Library \cite{GAP4}.
 These Wedderburn decompositions can be computed with the command \texttt{WedderburnDecompositionInfo} of the {\sf GAP} package {\sf Wedderga} \cite{Wedderga4.10.5}. 
The first two decompositions show that
\begin{eqnarray*}
\Q G_3 &\cong& (2\Q) \otimes_{\Q} (2\Q \oplus \Q(\sqrt{-1}) \oplus M_2(\Q))) \\
\Q G_4 &\cong& (2\Q) \otimes_{\Q} (2\Q \oplus \Q(\sqrt{-1}) \oplus \HQ(\Q)))
\end{eqnarray*}
While $2\Q \cong \Q C_2$, neither of the right hand factors is  isomorphic to a group algebra, because the Wedderburn decompositions of the non-abelian groups of order $8$ are
$$
\Q D_8 = 4\Q \oplus M_2(\Q) \qand \Q Q_8 = 4\Q \oplus \HQ(\Q).$$
 \bibliographystyle{amsalpha}
 \bibliography{MIP}

\newcommand{\etalchar}[1]{$^{#1}$}
\providecommand{\bysame}{\leavevmode\hbox to3em{\hrulefill}\thinspace}
\providecommand{\MR}{\relax\ifhmode\unskip\space\fi MR }
% \MRhref is called by the amsart/book/proc definition of \MR.
\providecommand{\MRhref}[2]{%
  \href{http://www.ams.org/mathscinet-getitem?mr=#1}{#2}
}
\providecommand{\href}[2]{#2}
\begin{thebibliography}{BBCH{\etalchar{+}}24}

\bibitem[BBCH{\etalchar{+}}24]{Wedderga4.10.5}
G.~K. Bakshi, O.~Broche~Cristo, A.~Herman, O.~Konovalov, S.~Maheshwary,
  G.~Olteanu, A.~Olivieri, Á. del Río, and I.~Van~Gelder, \emph{{Wedderga},
  wedderburn decomposition of group algebras, {V}ersion 4.10.5}, \href
  {https://gap-packages.github.io/wedderga}
  {\texttt{https://gap-packages.github.io/}\discretionary
  {}{}{}\texttt{wedderga}}, Feb 2024, Refereed GAP package.

\bibitem[CK95]{CarlsonKovacs}
J.~F. Carlson and L.~G. Kov{\'a}cs, \emph{Tensor factorizations of group
  algebras and modules}, J. Algebra \textbf{175} (1995), no.~1, 385--407
  (English).

\bibitem[GAP24]{GAP4}
The GAP~Group, \emph{{GAP -- Groups, Algorithms, and Programming, Version
  4.13.1}}, 2024.

\bibitem[GL24]{GL2022}
D.~García-Lucas, \emph{The modular isomorphism problem and abelian direct
  factors}, Mediterr. J. Math. \textbf{21} (2024), Article 18.

\bibitem[Hor86]{Horst1986}
C.~Horst, \emph{A cancellation theorem for artinian local algebras.},
  Mathematische Annalen \textbf{276} (1986), 657--662.

\bibitem[K{\"u}l91]{Kulshammer1991}
B.~K{\"u}lshammer, \emph{Lectures on block theory}, London Mathematical Society
  Lecture Note Series, vol. 161, Cambridge University Press, Cambridge, 1991.

\bibitem[Laz54]{Lazard1953}
M.~Lazard, \emph{Sur les groupes nilpotents et les anneaux de {Lie}}, Annales
  scientifiques de l'\'Ecole Normale Sup\'erieure \textbf{3e s{\'e}rie, 71}
  (1954), no.~2, 101--190 (fr). \MR{19,529b}

\bibitem[Lin18]{LinckelmannBook1}
M.~Linckelmann, \emph{The block theory of finite group algebras. vol.1}, London
  mathematical society student texts 91 ; 92, Cambridge University Press, 2018.

\bibitem[MSS23]{MSS23}
Leo Margolis, Taro Sakurai, and Mima Stanojkovski, \emph{Abelian invariants and
  a reduction theorem for the modular isomorphism problem}, J. Algebra
  \textbf{636} (2023), 1--27.

\bibitem[Pas77]{Pas77}
D.~S. Passman, \emph{The algebraic structure of group rings}, Pure and Applied
  Mathematics, Wiley-Interscience [John Wiley \& Sons], New York-London-Sydney,
  1977.

\bibitem[San85]{Sandling85}
R.~Sandling, \emph{The isomorphism problem for group rings: a survey}, Orders
  and their applications ({O}berwolfach, 1984), Lecture Notes in Math., vol.
  1142, Springer, Berlin, 1985, pp.~256--288.

\bibitem[Seh78]{Seh78}
S.~K. Sehgal, \emph{Topics in group rings}, Monographs and Textbooks in Pure
  and Applied Math., vol.~50, Marcel Dekker, Inc., New York, 1978.

\end{thebibliography}

\end{document}